\documentclass[11pt]{amsart}
\usepackage{geometry}
\usepackage{enumitem}
\usepackage{amstext}
\usepackage{amsthm}
\usepackage{amssymb}

\makeatletter
\theoremstyle{plain}
\newtheorem{thm}{\protect\theoremname}
 \theoremstyle{definition}
 \newtheorem*{defn*}{\protect\definitionname}
 \theoremstyle{definition}
 \newtheorem*{notation}{Notation}
  \theoremstyle{plain}
  \newtheorem{prop}[thm]{\protect\propositionname}
 \newlist{casenv}{enumerate}{4}
 \setlist[casenv]{leftmargin=*,align=left,widest={iiii}}
 \setlist[casenv,1]{label={{\itshape\ \casename} \arabic*.},ref=\arabic*}
 \setlist[casenv,2]{label={{\itshape\ \casename} \roman*.},ref=\roman*}
 \setlist[casenv,3]{label={{\itshape\ \casename\ \alph*.}},ref=\alph*}
 \setlist[casenv,4]{label={{\itshape\ \casename} \arabic*.},ref=\arabic*}
  \theoremstyle{remark}
  \newtheorem*{rem*}{\protect\remarkname}
  \theoremstyle{plain}
  \newtheorem{cor}[thm]{\protect\corollaryname}
  \theoremstyle{plain}
  \newtheorem{lem}[thm]{\protect\lemmaname}

\makeatother

  \providecommand{\corollaryname}{Corollary}
  \providecommand{\definitionname}{Definition}
  \providecommand{\lemmaname}{Lemma}
  \providecommand{\propositionname}{Proposition}
  \providecommand{\remarkname}{Remark}
 \providecommand{\casename}{Case}
\providecommand{\theoremname}{Theorem}

\def\Xint#1{\mathchoice
   {\XXint\displaystyle\textstyle{#1}}%
   {\XXint\textstyle\scriptstyle{#1}}%
   {\XXint\scriptstyle\scriptscriptstyle{#1}}%
   {\XXint\scriptscriptstyle\scriptscriptstyle{#1}}%
   \!\int}
\def\XXint#1#2#3{{\setbox0=\hbox{$#1{#2#3}{\int}$}
     \vcenter{\hbox{$#2#3$}}\kern-.5\wd0}}
\def\fint{\Xint-}

\global\long\def\osc{\operatornamewithlimits{osc}}
\global\long\def\Diff{\operatorname{Diff}}

\global\long\def\co{\colon\thinspace}

\subjclass[2000]{53C44; 35K45, 35K59, 58J35}

\title{A Heat Flow for Diffeomorphisms of Flat Tori}
\author{Ben Andrews}
\address{Mathematical Sciences Institute, Australian National University,
ACT 2601, Australia}
\email{ben.andrews@anu.edu.au}
\urladdr{http://maths.anu.edu.au/\textasciitilde{}andrews/}
\thanks{This research was partially supported by the Australian Laureate
Fellowship grant FL150100126 of the Australian Research Council.}

\author{Anthony Carapetis}
\email{anthony.carapetis@gmail.com}
\urladdr{http://a.carapetis.com/math}

\begin{document}
\begin{abstract}
In this paper we study the parabolic evolution equation
$\partial_{t}u=(\left|Du\right|^{2}+2\left|\det Du\right|)^{-1}\Delta u$,
where $u\co M\times[0,\infty)\to N$ is an evolving map between compact
flat surfaces. We use a tensor maximum principle for the induced metric
$Du^{T}Du$ to establish two-sided bounds on the singular values of
$Du$, which shows that unlike harmonic map heat flow, this flow preserves
diffeomorphisms. A change of variables for $Du$ then allows us to
establish a $C^{\alpha}$ estimate for the coefficient of the tension
field, and thus (thanks to the quasilinear structure and the Schauder
estimates) we get full regularity and long-time existence. We conclude
with some energy estimates to show convergence to an affine diffeomorphism.
\end{abstract}

\maketitle

\section{Introduction}

Unlike many similar problems in geometry, there is no widely applicable
heat-flow technique to produce harmonic (or otherwise geometrically
nice) representatives of diffeomorphisms. Unfortunately harmonic map
heat flow does not work in dimensions greater than one: by choosing
the right 2-jet at a point, one can arrange for the derivative to
become singular a short time afterwards, so the flow does not stay
inside the space of diffeomorphisms. In this paper we define and study
a heat flow for diffeomorphisms of flat compact surfaces. Let $M,N$
be oriented flat tori; i.e. $M=\mathbb{R}^{2}/\Gamma_{1},N=\mathbb{R}^{2}/\Gamma_{2}$
for some integer lattices\footnote{That is, additive subgroups of $\mathbb{R}^{2}$ isomorphic to $\mathbb{Z}^{2}$.}
$\Gamma_{1}$ and $\Gamma_{2}$ acting by translation. Without loss
of generality we can assume the initial data $u_{0}\in\Diff\left(M,N\right)$
is orientation-preserving, i.e. $\det Du>0$ when we choose oriented
coordinates on both $M$ and $N$. Then we consider the Cauchy problem
for a family of maps $u\co M\times[0,T)\to N$ defined by
\begin{equation}
\left\{ \begin{array}{lc}
{\displaystyle \frac{\partial u}{\partial t}=\frac{\Delta u}{\left|Du\right|^{2}+2\det Du}}\\
u\left(x,0\right)=u_{0}
\end{array}\right.\label{eq:main-flow}
\end{equation}
where $\Delta u$ is the tension field (or simply the component-wise
Laplacian in Cartesian coordinates). In this notation the denominator
is somewhat opaque - see the next section for a more geometrically
enlightening formula. Our main result is the following:
\begin{thm}
There exists a solution $u\in C^{\infty}\left(M\times[0,\infty),N\right)$
of (\ref{eq:main-flow}) such that each $u\left(\cdot,t\right)$ is
a diffeomorphism, $u\left(0,t\right)=u_{0}$ and $u\left(\cdot,t\right)$
converges smoothly to a harmonic map $u_{\infty}$ as $t\to\infty$.\label{thm:main-theorem-T2}
\end{thm}
We will frequently (and often silently) swap between thinking of $u\left(\cdot,t\right)$
as a map between compact manifolds $M\to N$ and as the corresponding
map $\mathbb{R}^{2}\to\mathbb{R}^{2}$ between universal covers -
the former is conceptually what we care about but the latter means
we can choose Cartesian coordinates and work with a very concrete
global PDE system.
\begin{notation}
\noindent From here forward $u$ will always be a solution of (\ref{eq:main-flow}),
$F=(\left|Du\right|^{2}+2\left|\det Du\right|)^{-1}$ the \emph{diffusion
coefficient }and $P=\partial_{t}-F\left(Du\right)\Delta$ (acting
on scalar functions) the linear parabolic operator associated with
(\ref{eq:main-flow}). We will use subscript notation $u_{jk}^{i}={\displaystyle \frac{\partial^{2}u^{i}}{\partial x^{j}\partial x^{k}}}$
for the derivatives of $u$ ($x$ \emph{any }Cartesian coordinate
system on $\mathbb{R}^{2}$) and the summation convention for repeated
indices. Symmetrisation is denoted by $T_{(ij)}=\frac{1}{2}\left(T_{ij}+T_{ji}\right)$.
The derivative $Du$ of maps/functions defined
on $M\times[0,T)$ refers to the spatial components alone - time derivatives
will only appear explicitly as $\partial_{t}$.
\end{notation}

\section{Gradient bounds and Preservation of Diffeomorphisms}

Since the initial map is a diffeomorphism and any homotopy preserves
degree, the only way for the map to no longer be a diffeomorphism
at a later time is if it is not even a local diffeomorphism; that
is, if $Du$ becomes singular at some point. Thus we can establish
that the flow preserves diffeomorphisms by preserving the invertibility
of $Du$ using the maximum principle. In particular we will preserve
bounds on the singular values $\lambda_{i}$ in the singular value
decomposition
\begin{equation}
Du\left(e_{i}\right)=\lambda_{i}v_{i},\,i=1,2.\label{eq:svd}
\end{equation}
Recall here that $e_{i},v_{i}$ are orthonormal bases for $T_{x}M$,
$T_{u(x)}N$ respectively and $\lambda_{i}\ge0$ are the square roots
of the eigenvalues of the induced metric $h=u^{*}g_{N}=Du^{T}Du$.
Note that $\lambda_{1}\lambda_{2}=\left|\det Du\right|$ and $\lambda_{1}^{2}+\lambda_{2}^{2}=\left|Du\right|^{2}$,
so our flow can be written more naturally as:
\[
\frac{\partial u}{\partial t}=\frac{\Delta u}{\left(\lambda_{1}+\lambda_{2}\right)^{2}}
\]

We need lower bounds on $\lambda_{i}$ to preserve diffeomorphisms
and upper bounds to establish the H\"older estimate in the next section.
Rather than studying the evolution of the singular values directly
(which requires some hairy computations and gets unwieldy at singular
points $\lambda_{1}=\lambda_{2}$ where regularity fails), we will
focus on the evolution of $h_{ij}=u_{i}^{\alpha}u_{j}^{\alpha}$.
Since the eigenvalues of $h$ are $\lambda_{1}^{2}$ and $\lambda_{2}^{2}$,
preserving bounds $m\le\lambda_{i}\le M$ is equivalent to preserving
the inequality $m^{2}\delta\le h\le M^{2}\delta$ of quadratic forms.
\begin{prop}
\label{prop:c1-estimate}If the singular values of $Du$ satisfy $m\le\lambda_{i}\le M$
at the initial time, then this inequality persists for all later times.\end{prop}
\begin{proof}
We will show $h\ge m^{2}\delta$ is preserved using \cite[Theorem 3.2]{andrews-pinch},
a refinement of Hamilton's maximum principle for tensors. The argument
for the upper bound $h\le M^{2}\delta$ is very similar but easier,
since there the negative-definite term in (\ref{eq:reaction-N}) below
is a help rather than a hindrance. Let $S=h-m^{2}\delta$ so that
we are trying to preserve the non-negative definiteness of $S$. Then
after some computation we arrive at the evolution equation
\[
\partial_{t}S_{ij}=F\Delta S_{ij}+N_{ij}
\]
where: 
\begin{equation}
N_{ij}=-2F\left(u_{ki}^{\alpha}u_{kj}^{\alpha}+\frac{2u_{k(i}^{k}u_{j)}^{\alpha}\Delta u^{\alpha}}{\lambda_{1}+\lambda_{2}}\right)\label{eq:reaction-N}
\end{equation}

If $S\ge0$ and $S(p)(v,v)=0$, we can choose Cartesian coordinates
for the domain and target so that $v=\partial_{1}$ and $Du=\text{diag}\left(\lambda_{1},\lambda_{2}\right)$
with $\lambda_{2}\ge\lambda_{1}=m$. Then at $p$ we have $S=h-\lambda_{1}^{2}\delta=\text{diag}\left(0,\lambda_{2}^{2}-\lambda_{1}^{2}\right)$;
and since $h(\partial_{1},\partial_{1})$ is at a spatial minimum
we have $\partial_{i}h_{11}=2\lambda_{1}u_{1i}^{1}=0$ and thus $u_{1i}^{1}=0$.
The other derivatives $\partial_{i}S=\partial_{i}h$ are given at
$p$ by:
\[
\partial_{1}S=\left(\begin{array}{cc}
0 & \lambda_{2}u_{11}^{2}\\
\lambda_{2}u_{11}^{2} & 2\lambda_{2}u_{12}^{2}
\end{array}\right)\qquad\partial_{2}S=\left(\begin{array}{cc}
0 & \lambda_{1}u_{22}^{1}+\lambda_{2}u_{12}^{2}\\
\lambda_{1}u_{22}^{1}+\lambda_{2}u_{12}^{2} & 2\lambda_{2}u_{22}^{2}
\end{array}\right)
\]
To preserve the inequality it suffices to show that 
\[
Q=N_{11}+2F\sup_{\Gamma\in\mathbb{R}^{2\times2}}\left(2\Gamma_{k}^{2}\partial_{k}S_{12}-\Gamma_{k}^{2}\Gamma_{k}^{2}S_{22}\right)\ge0
\]
at this point\footnote{The supremum term here comes from noting that at $p$, the scalar
function $x\mapsto S(V_{x},V_{x})$ must be at a minimum for \emph{any
}vector field $V$ extending $v$. Choosing $V=v+\Gamma x$ and computing
the evolution equation of $S(V,V)$ rather than $S(v,v)=S_{11}$ gives
the extra term.}.
\begin{casenv}
\item If $\lambda_{1}=\lambda_{2}$ then $S_{22}=0$ and $h(\partial_{2},\partial_{2}$)
is also at a spatial minimum, so we also have $u_{2i}^{2}=0$. If
$u_{11}^{2}=u_{22}^{1}=0$ then the supremum and the reaction term
$N_{11}$ are both zero. Otherwise, we are taking the supremum of
a non-constant linear function and thus get $Q=\infty\ge0$. 
\item If $\lambda_{1}\ne\lambda_{2}$, $S_{22}=\lambda_{2}^{2}-\lambda_{1}^{2}$
is positive, so the expression being maximised is a quadratic polynomial
in $\Gamma_{1}^{2}$ and $\Gamma_{2}^{2}$ with negative leading coefficients.
Thus it achieves a unique maximum at\footnote{Remark: this optimal $\Gamma$ is achieved at the singular vector
field $V=v+\Gamma x=e_{1}+O(|x|^{2})$; so the $Q$ we get in this
case is actually $P(\lambda_{1}^{2})$.} $\Gamma_{k}^{2}=\partial_{k}S_{12}/S_{22}$, where it is equal to:
\[
\frac{\left(\partial_{1}S_{12}\right)^{2}+\left(\partial_{2}S_{12}\right)^{2}}{S_{22}}=\frac{\lambda_{2}^{2}u_{11}^{2}u_{11}^{2}+\lambda_{1}^{2}u_{22}^{1}u_{22}^{1}+\lambda_{2}^{2}u_{12}^{2}u_{12}^{2}+2\lambda_{1}\lambda_{2}u_{22}^{1}u_{12}^{2}}{\lambda_{2}^{2}-\lambda_{1}^{2}}
\]
Combining this with $N_{11}$ we get some nice cancellations resulting
in
\[
Q=\frac{2\lambda_{1}^{2}F}{\lambda_{2}^{2}-\lambda_{1}^{2}}\left(\left(u_{11}^{2}\right)^{2}+\left(u_{12}^{2}+u_{22}^{1}\right)^{2}\right)\ge0
\]
as required.
\end{casenv}
\end{proof}
\begin{rem*}
There are other choices of $F$ (or anisotropic flows $\partial_{t}u=a^{ij}\left(Du\right)\nabla_{i}\nabla_{j}u$)
satisfying this result - rather than requiring perfect cancellation
of the $u_{22}^{1}u_{12}^{2}$ terms between $N_{11}$ and the $\Gamma$
contribution to form the perfect square, we could be a little more
permissive and require some differential inequality on the coefficients
that leads to the desired definiteness. We will not discuss these
other flows, however, as the regularity estimate in the next section
works only for (\ref{eq:main-flow}).
\end{rem*}
Since the domain manifold is compact and the initial data is a diffeomorphism,
we know there must be \emph{some }$m,M$ for which these bounds hold.
To keep things simple from here on out, we will consolidate these
persistent derivative bounds into a single constant:
\begin{cor}
There is some $\Lambda>0$ depending only on the initial data $u_{0}$
such that $\lambda_{1},\lambda_{2},\left|Du\right|,\lambda_{1}+\lambda_{2}$
are all in $\left[\Lambda^{-1},\Lambda\right]$ for all time.
\end{cor}

\section{Regularity Estimate}

Since we are dealing with a quasilinear PDE system, there is very
little general theory available; so to get a H\"{o}lder estimate on
the coefficient $F$ we will need to exploit the particular form of
our system. To get this estimate we need to define some new quantities:
fix Cartesian coordinates on both the domain and target, and let $\theta\in S^{1}$
and $r>0$ be defined by
\begin{eqnarray*}
r\cos\theta & = & u_{1}^{1}+u_{2}^{2},\\
r\sin\theta & = & u_{1}^{2}-u_{2}^{1}.
\end{eqnarray*}
One can check that $r=\lambda_{1}+\lambda_{2}$ and $\theta$ is the
``rotational component'' of $Du$; that is, the angle between the
singular frames $e$ and $v$. Since $F=r^{-2}$ and Prop \ref{prop:c1-estimate}
gives us time-independent bounds above and below on $r$, to get a
$C^{\alpha}$ estimate for $F$ it suffices to get one for $r$. This
is our goal for this section. We start with the result that explains
why we are interested in $\theta$ along with $r$:
\begin{lem}
The evolution of the quantities $r,\theta$ forms a closed system
of PDE given by:\footnote{$Dr\times D\theta$ denotes the 2D cross product $\partial_{1}r\,\partial_{2}\theta-\partial_{2}r\,\partial_{1}\theta$.}
\begin{eqnarray*}
Pr & = & -\frac{\left|D\theta\right|^{2}}{r}-\frac{2\left|Dr\right|^{2}}{r^{3}}+\frac{Dr\times D\theta}{r^{2}}\\
P\theta & = & 0
\end{eqnarray*}
\end{lem}
\begin{proof}
As with the previous evolution equations, differentiate (\ref{eq:main-flow})
to get the evolution of $Du$, use this to find the evolution of $r,\theta$
and do a lot of algebraic simplification. This result relies strongly
on the particular $F$ we have chosen, and is related to the fact
that our choice of $F$ makes the flow preserve maps with symmetric
$Du$.
\end{proof}
The first thing we notice is that the evolution of $\theta$ decouples
entirely - it satisfies a uniformly parabolic PDE in general form,
so we can apply the standard Krylov-Safanov estimate \cite[Corollary 7.41]{Lieberman-book}:
\begin{cor}
$\theta$ belongs to some parabolic H\"{o}lder space, with exponent
and norm depending only on $\Lambda$.
\end{cor}
Our estimate for $r$ is inspired by the general H\"{o}lder gradient
estimate for quasilinear PDE (see Lieberman \cite[Theorem 12.3]{Lieberman-book},
originally due to Ladyzhenskaya and Uraltseva \cite{LU1}) which perturbs
$\partial_{i}u$ by $\left|du\right|^{2}$ in order to obtain super-
and subsolutions of a divergence-form equation, allowing the application
of a weak Harnack inequality. We will instead perturb $r$ by $\theta$,
with the H\"{o}lder continuity of $\theta$ being key to getting the
same regularity for $r$.
\begin{notation}
\noindent For $X_{0}=\left(x_{0},t_{0}\right)\in\mathbb{T}^{2}\times[0,T)$
we define the parabolic neighbourhood 
\[
Q\left(X_{0},R\right)=\left\{ \left(x,t\right)\in M\times[0,T)\co d\left(x,x_{0}\right)<R,\left|t-t_{0}\right|<R^{2},t<t_{0}\right\} .
\]
For $\alpha\in(0,1]$, the parabolic $\alpha$-H\"{o}lder space is
then defined by the norm 
\[
\left\Vert f\right\Vert _{0;\alpha}=\sup\left|f\right|+\sup\left\{ \frac{\osc_{Q\left(X,R\right)}f}{R^{\alpha}}\middle|X\in M\times[0,T),R>0\right\} .
\]

\end{notation}
We will denote averages by
\[
\fint_{E}f=\frac{1}{\mu\left(E\right)}\int_{E}f\,d\mu
\]
where $d\mu=dA\,dt$ is the standard measure on $M\times[0,T)$, and
often use the shorthand $\sup_{R}=\sup_{Q\left(X_{0},R\right)}$ (similarly
$\inf_{R},\osc_{R}$) since almost all the neighbourhoods we are working
with in this estimate have a common centre. The relation $A\lesssim B$
means $A\le CB$ for some positive constant $C$ depending only on
$\Lambda$.
\begin{prop}
The diffusion coefficient $F\left(x,t\right)=\left(\lambda_{1}\left(x,t\right)+\lambda_{2}\left(x,t\right)\right)^{-2}$
is bounded in the parabolic H\"{o}lder space over $M\times[0,T)$,
with norm depending only on $\Lambda$.\end{prop}
\begin{proof}
This is a local estimate, so fix a point $X_{0}=\left(x_{0},t_{0}\right)$
now and let $\phi=\theta-\theta\left(X_{0}\right)$ (so $D\theta=D\phi$).
Since we know $r\in\left[\Lambda^{-1},\Lambda\right]$, oscillations
of $r^{q}$ are comparable to those of $r$ for any $q>0$: to be
precise we have 
\begin{equation}
\frac{\left|s^{q}-r^{q}\right|}{\left|s-r\right|}\in\left[q\Lambda^{-\left|q-1\right|},q\Lambda^{\left|q-1\right|}\right]\label{eq:exponent-comparability}
\end{equation}
for all $r,s\in\left[\Lambda^{-1},\Lambda\right]$. Thus we can take
powers of $r$ to get more favourable evolution equations: we will
work with the two functions $r^{2}$ and $w=r^{p}-c\phi^{2}$, where
$p$ and $c$ will be chosen later to make $w$ a supersolution. Letting
\[
L_{d}\co f\mapsto\partial_{t}f-\text{div}\left(r^{-2}Df\right)
\]
be our associated linear divergence-form operator, we compute
\begin{align*}
L_{d}r^{2} & =2r^{-1}Dr\times D\theta-2\left|D\theta\right|^{2}-2r^{-2}\left|Dr\right|^{2}\\
L_{d}w & =pr^{p-2}\left(r^{-1}Dr\times D\theta-\left|D\theta\right|^{2}+\left(1-p\right)r^{-2}\left|Dr\right|^{2}\right)\\
 & +2cr^{-2}\left|D\theta\right|^{2}-4c\phi r^{-3}\left\langle Dr,D\theta\right\rangle .
\end{align*}
The Cauchy--Schwarz and Peter--Paul inequalities allow us to estimate
\[
\left|Dr\times D\theta\right|,\left|\left\langle Dr,D\theta\right\rangle \right|\le\frac{1}{2r}\left|Dr^{2}\right|+\frac{r}{2}\left|D\theta\right|^{2}
\]
and thus we immediately get 
\[
L_{d}r^{2}\le-r^{-2}\left|Dr\right|^{2}-\left|D\theta\right|^{2}\le0.
\]
Getting a supersolution is more difficult: using the same estimates
for the cross terms we have
\begin{align*}
L_{d}\left(r^{p}-c\phi^{2}\right) & \ge\left(-\frac{1}{2}pr^{p-4}+p\left(1-p\right)r^{p-4}-2c\phi r^{-4}\right)\left|Dr\right|^{2}\\
 & +\left(-\frac{1}{2}pr^{p-2}-pr^{p-2}+2cr^{-2}-2c\phi r^{-2}\right)\left|D\theta\right|^{2}.
\end{align*}
For this to be non-negative we need the two inequalities
\begin{align*}
p\left(\frac{1}{2}-p\right)r^{p}-2c\phi & \ge0\\
\tag*{and}2c\left(1-\phi\right)-\frac{3}{2}pr^{p} & \ge0.
\end{align*}
In particular these are both satisfied if $\left|\phi\right|<\delta<1$
and 
\[
\frac{3}{4}\frac{p\Lambda^{p}}{1-\delta}\le c\le\frac{p\left(\frac{1}{2}-p\right)}{2\delta}\Lambda^{-p},
\]
and we can choose such a $c$ if and only if 
\[
\frac{3\delta}{1-\delta}\Lambda^{2p}\le1-2p.
\]
Thus taking $\delta<1/4$ we can satisfy this with some $p\in(0,\frac{1}{2})$.
By the uniform H\"{o}lder continuity of $\theta$ we can choose $R_{0}$
depending only on $\Lambda$ so that $\left|\phi\right|<\delta$ on
$Q\left(X_{0},4R_{0}\right)$, and thus we have $L_{d}r^{2}\le0$
and $L_{d}w\ge0$ on this set. Now let $R\in\left(0,R_{0}\right)$
be arbitrary and apply the Weak Harnack estimate \cite[Theorem 6.18]{Lieberman-book}
to the non-negative supersolutions $\sup_{4R}r^{2}-r^{2}$ and $w-\inf_{4R}w$
on $Q\left(4R\right)$, which produces estimates
\begin{align}
\fint_{\Theta\left(R\right)}\left[\sup_{4R}r^{2}-r^{2}\right] & \le C\left(\sup_{4R}r^{2}-\sup_{R}r^{2}\right)\label{eq:harnack-application-1}\\
\fint_{\Theta\left(R\right)}\left[\left(r^{p}-c\phi^{2}\right)-\inf_{4R}\left(r^{p}-c\phi^{2}\right)\right] & \le C\left(\inf_{R}\left(r^{p}-c\phi^{2}\right)-\inf_{4R}\left(r^{p}-c\phi^{2}\right)\right),\label{eq:harnack-application-2}
\end{align}
where $\Theta\left(R\right)=Q\left(\left(x_{0},t_{0}-4R^{2}\right),R\right)\subset Q\left(X_{0},4R\right)$
covers the same spatial region as $Q\left(R\right)$ but is disjoint
in time and $C>0$ depends only on $\Lambda$. Using (\ref{eq:exponent-comparability})
we get:
\begin{gather*}
\fint_{\Theta\left(R\right)}\left[\sup_{4R}r^{2}-r^{2}\right]\ge2\Lambda^{-1}\fint_{\Theta\left(R\right)}\left[\sup_{4R}r-r\right]\\
\fint_{\Theta\left(R\right)}\left[\left(r^{p}-c\phi^{2}\right)-\inf_{4R}\left(r^{p}-c\phi^{2}\right)\right]\ge p\Lambda^{p-1}\fint_{\Theta\left(R\right)}\left[r-\inf_{4R}r\right]
\end{gather*}
Similarly on the other side we have:
\begin{gather*}
\sup_{4R}r^{2}-\sup_{R}r^{2}\le2\Lambda\left(\sup_{4R}r-\sup_{R}r\right)\\
\inf_{R}\left(r^{p}-c\phi^{2}\right)-\inf_{4R}\left(r^{p}-c\phi^{2}\right)\le p\Lambda^{1-p}\left(\inf_{R}r-\inf_{4R}r\right)+\sup_{4R}c\phi^{2}
\end{gather*}
Using these estimates in (\ref{eq:harnack-application-1}) and (\ref{eq:harnack-application-2})
and multiplying by appropriate factors yields 
\begin{gather*}
\fint_{\Theta\left(R\right)}\left[\sup_{4R}r-r\right]\le C\Lambda^{2}\left(\sup_{4R}r-\sup_{R}r\right)\\
\tag*{and}\fint_{\Theta\left(R\right)}\left[r-\inf_{4R}r\right]\le C\Lambda^{2-2p}\left(\inf_{R}r-\inf_{4R}r\right)+\frac{1}{p}\Lambda^{1-p}\sup c\phi^{2}.
\end{gather*}
Noting $\Lambda^{-2p}\le1$ and adding these together we find 
\[
\osc_{4R}r\le C_{1}(\osc_{4R}r-\osc_{R}r+\sup_{4R}\phi^{2})
\]
where the new constant $C_{1}=\max\left(C\Lambda^{2},\frac{c}{p}\Lambda^{1-p},1\right)$
depends only on $\Lambda$. Writing this as
\[
\osc_{R}r\le\left(1-C_{1}^{-1}\right)\osc_{4R}r+\sup_{4R}\phi^{2}
\]
and applying Lemma 8.23 of Gilbarg--Trudinger \cite{GT-EPDE} with
$\sigma\left(R\right)=\sup_{R}\phi^{2}$ yields the estimate
\[
\osc_{R}r\le C_{2}\left(\left(\frac{R}{R_{0}}\right)^{\alpha}\osc_{R_{0}}r+\sup_{R^{\mu}R_{0}^{1-\mu}}\phi^{2}\right)
\]
where $C_{2}>0$ and $\alpha\in\left(0,1\right)$ depend only on $\Lambda$
and a freely chosen $\mu\in\left(0,1\right)$. The H\"{o}lder bound
on $\theta$ implies $\sup_{\rho}\phi^{2}\lesssim\rho^{\beta}$ for
some $\beta\in\left(0,1\right)$, and thus we have (for $R<R_{0}$)
\[
\osc_{R}r\lesssim R^{\alpha}+R^{\mu\beta}\lesssim R^{\min\left(\alpha,\mu\beta\right)},
\]
i.e. a H\"{o}lder bound on $r$ with constant and exponent depending
only on $\Lambda$ (since $\beta$ and $R_{0}$ depend only on $\Lambda$).
Since $F=r^{-2}$ and we already know $r$ is restricted to $\left[\Lambda^{-1},\Lambda\right]$,
we get the same result for $X\mapsto F\left(Du\left(X\right)\right)$.
\end{proof}
A standard bootstrap argument using the parabolic Schauder estimate
now gives $C^{\alpha}$ regularity of all derivatives; so we get the
estimate we need for long-time existence:
\begin{cor}
All derivatives of $u$ are bounded on $M\times[0,T)$.\label{cor:longtime}\end{cor}
\begin{proof}
Repeatedly differentiating the evolution equation $\partial_{t}u=F\Delta u$
gives
\[
P\left(D^{k}u\right)=\partial_{t}D^{k}u-F\Delta D^{k}u=D^{k}F*D^{2}u+\cdots+DF*D^{k+1}u
\]
where $A*B$ denotes an arbitrary contraction of $A\otimes B$. By
our control on $Du$ we know that $\left\Vert F\right\Vert _{k;\alpha}\lesssim\left\Vert u\right\Vert _{k+1;\alpha}$,
so $\left\Vert P\left(D^{k}u\right)\right\Vert _{\alpha}\lesssim\left\Vert u\right\Vert _{k+1;\alpha}$.
Combining this with the Schauder estimate we see that $\left\Vert u\right\Vert _{k+1;\alpha}<\infty$
implies $\left\Vert u\right\Vert _{k+2;\alpha}<\infty$ for all $k\ge1$;
so once we have $C^{2;\alpha}$ control we get bounds on all derivatives.
Applying the Schauder estimate to the components of $u$ we see that
our $C^{\alpha}$ control of the coefficients implies this $C^{2;\alpha}$
control on $u$.
\end{proof}

\section{Existence and Convergence}

We come now to the proof of Theorem \ref{thm:main-theorem-T2}. The
first technicality to address is short-time existence, which is (as
is usually the case) fairly straightforward for our flow. 
\begin{lem}
For any $t_{0}\in\mathbb{R}$ and $u_{0}\in{\rm Diff}\left(M,N\right)$,
there exists some time $\epsilon>0$ and a smooth solution $u\co M\times[t_{0},t_{0}+\epsilon)\to N$
of $\partial_{t}u=F\Delta u$ such that $u\left(\cdot,t_{0}\right)=u_{0}$.\label{lem:short-time-existence}\end{lem}
\begin{proof}
We refer the reader to Baker \cite[Main Theorem 1]{bakerthesis} for
a local existence theorem for solutions of nonlinear parabolic systems.
The conditions are easy to check - since the initial data is a diffeomorphism
of a compact manifold we have $\Lambda^{-2}\le F|_{t=0}\le\Lambda^{2}$,
so we immediately satisfy the Legendre-Hadamard condition and can
modify $F$ on $B\left(0,\Lambda^{-1}/2\right)$ (i.e. away from the
image of $Du$) to make it continuously differentiable without changing
the dynamics.
\end{proof}
We now have all the ingredients we need for the standard long-time
existence argument:
\begin{prop}
Given any $u_{0}\in{\rm Diff}(M,N)$, there is a smooth solution $u\co M\times[0,\infty)\to N$
of (\ref{eq:main-flow}) such that $u(\cdot,0)=u_{0}$ and each $u(\cdot,t)$
is a diffeomorphism.\label{prop:long-time-existence}\end{prop}
\begin{proof}
Let $T\in[0,\infty]$ denote the largest time such that there is a
smooth solution on $[0,T)$. By Lemma \ref{lem:short-time-existence}
we know that $T>0$. If $0<T<\infty$, Corollary \ref{cor:longtime}
tells us that all derivatives of $u$ are equicontinuous, so by the
Arzela-Ascoli theorem we can smoothly extend the solution to $[0,T]$.
But then $u\left(\cdot,T\right)$ is also a diffeomorphism (by continuity
of the derivative and the strong lower bound $\lambda_{i}>\Lambda^{-1}$),
so by Lemma \ref{lem:short-time-existence} again we can extend our
solution to some $[0,T+\epsilon)$, which contradicts the assumption
that $T$ is maximal. Thus $T=\infty$.\end{proof}
\begin{prop}
For any solution $u\co M\times[0,\infty)\to N$ of (\ref{eq:main-flow})
there is a sequence of times $t_{k}\to\infty$ such that $u\left(t_{k}\right)$
converges to a harmonic diffeomorphism; that is, a map of the form
$u\left(x\right)=Ax+y$ for some $y\in\mathbb{R}^{2}$ and $A\in GL\left(2,\mathbb{R}\right)$
sending $\Gamma_{1}$ onto $\Gamma_{2}$.\end{prop}
\begin{proof}
First let's establish the characterisation of harmonic diffeomorphisms
$M\to N$. If $u$ is harmonic then integration by parts over a fundamental
domain of $M$ yields 
\[
\int\left|D^{2}u\right|^{2}=\int\left(\Delta u\right)^{2}=0;
\]
so $u$ is an affine function $u\left(x\right)=Ax+y$ when considered
as a map $\mathbb{R}^{2}\to\mathbb{R}^{2}$. In order for this to
descend to a the quotients $M\to N$ it must satisfy some periodicity
$u\left(x+z\right)=u\left(x\right)+Bz$ for all $x\in\mathbb{R}^{2}$,
$z\in\mathbb{Z}^{2}$, where $B\co\Gamma_{1}\to\Gamma_{2}$ is a homomorphism
of abelian groups; and putting these together we see that in fact
$y=u\left(0\right)$ and $A=B$. Since $u$ is a diffeomorphism, $u^{-1}$
must be of the same form; i.e. $A$ must be an isomorphism.

We can get convergence on a sequence of times by studying the Dirichlet
energy $E\left(u\right)=\frac{1}{2}\int\left|Du\right|^{2}$. Integration
by parts shows that $E\left(t\right)=E\left(u\left(\cdot,t\right)\right)$
is a positive non-increasing function, and thus it converges to some
limit $E_{\infty}$ as $t\to\infty$. Since $\int_{0}^{\infty}E'(t)dt=E_{\infty}-E\left(0\right)$
is finite and $E'(t)$ is nonpositive, we can extract a sequence of
times $t_{k}\to\infty$ such that $E'\left(t_{k}\right)\to0$. From
the evolution equation we then see that
\[
\frac{dE}{dt}=-\int F\left(\Delta u\right)^{2}\to0;
\]
so the lower bound on $F$ implies that $\Delta u\left(t_{k}\right)\to0$
in $L^{2}$. Since $u$ is bounded in every $C^{k}$ norm, all of
its derivatives are equicontinuous and thus we can use Arzela-Ascoli
to pass to a diagonal subsequence of times (replacing $t_{k}$ from
now on) on which $u\left(t_{k}\right)$ converges smoothly to a limit
$u_{\infty}$. In particular we have $\Delta u\left(t_{k}\right)\to\Delta u_{\infty}$
uniformly and thus in $L^{2}$ as well, so $\Delta u_{\infty}=0$.
\end{proof}
To improve this convergence on a sequence to convergence for all times,
we estimate the second-order energy $q\left(t\right)=\left\Vert D^{2}u\left(t\right)\right\Vert _{L^{2}}^{2}=\int\left|\Delta u\right|^{2}.$
\begin{lem}
There are positive constants $C_{0},C_{1}$ depending only on $\Lambda$
such that $q\left(t\right)$ satisfies
\[
\frac{dq}{dt}\le-C_{0}\left(1-C_{1}q\left(t\right)\right)\left\Vert D^{3}u\right\Vert _{L^{2}}^{2}.
\]
\end{lem}
\begin{proof}
From the evolution equation we can compute
\[
\frac{dq}{dt}=-2\int F\left(Du\right)\left|D^{3}u\right|^{2}-2\int\dot{F}\left(Du\right)*D^{2}u*D^{2}u*D^{3}u.
\]
where $\dot{F}$ is the matrix of derivatives $\partial F/\partial u_{j}^{i}$.
Estimating the second term with Peter-Paul as
\[
\left|\int\dot{F}\left(Du\right)*D^{2}u*D^{2}u*D^{3}u\right|\le\left\Vert \dot{F}\left(Du\right)\right\Vert _{L^{\infty}}\left(\epsilon\left\Vert \left|D^{2}u\right|^{2}\right\Vert _{L^{2}}^{2}+\frac{1}{\epsilon}\left\Vert D^{3}u\right\Vert _{L^{2}}^{2}\right)
\]
with $\epsilon=\Lambda^{2}\left\Vert \dot{F}\left(Du\right)\right\Vert _{L^{\infty}}$,
we get
\begin{eqnarray*}
\frac{dq}{dt} & \le & -\Lambda^{-2}\left\Vert D^{3}u\right\Vert _{L^{2}}^{2}+\Lambda^{2}\left\Vert \dot{F}\left(Du\right)\right\Vert _{L^{\infty}}^{2}\left\Vert D^{2}u\right\Vert _{L^{4}}^{4}.
\end{eqnarray*}
Applying the Gagliardo-Nirenberg interpolation inequality $\left\Vert f\right\Vert _{L^{4}}^{2}\le C\left\Vert f\right\Vert _{L^{2}}\left\Vert Df\right\Vert _{L^{2}}$
to $\left\Vert D^{2}u\right\Vert _{L^{4}}$, we arrive at
\[
\frac{dq}{dt}\le-\Lambda^{-2}\left\Vert D^{3}u\right\Vert _{L^{2}}^{2}\left(1-C\Lambda^{4}\left\Vert \dot{F}\left(Du\right)\right\Vert _{L^{\infty}}^{2}q\right)
\]
as desired. (Since $F$ is continuously differentiable on the compact
image of $Du$, $\left\Vert \dot{F}\left(Du\right)\right\Vert _{L^{\infty}}$
is a finite constant depending on $\Lambda$.)
\end{proof}
By the uniform subconvergence $D^{2}u\left(t_{k}\right)\to0$, there
is some time $t'$ at which $C_{1}q\left(t'\right)<\frac{1}{2}$,
and we see $q'\left(t\right)<0$ whenever this is true; so this inequality
is preserved for all time. Thus we have
\begin{cor}
For $t>t'$, $q$ satisfies
\begin{equation}
\frac{dq}{dt}<-\frac{C_{0}}{2}\left\Vert D^{3}u\right\Vert _{L^{2}}^{2}.\label{eq:dqdt-d3u2}
\end{equation}
\end{cor}
\begin{prop}
The quantity $q$ converges exponentially to zero: that is, $q\left(t\right)\le Ae^{-\omega t}$
for some $A,\omega>0$.\end{prop}
\begin{proof}
Applying the Poincar\'{e} inequality to $u_{jk}^{i}$ and summing,
we see $\left\Vert D^{2}u\right\Vert _{L^{2}}^{2}\lesssim\left\Vert D^{3}u\right\Vert _{L^{2}}^{2}$,
and thus for $t>t'$ we have
\[
\frac{dq}{dt}<-\omega q
\]
for a positive constant $\omega$. Comparison with the corresponding
ODE proves that $q\left(t\right)\le q\left(t'\right)e^{-\omega\left(t-t'\right)}$.
\end{proof}
As a consequence we obtain the full convergence, completing the proof
of Theorem \ref{thm:main-theorem-T2}.
\begin{prop}
The flow $u\left(t\right)$ converges smoothly to $u_{\infty}$ as
$t\to\infty$.\end{prop}
\begin{proof}
Using $\left\Vert \Delta u\right\Vert _{L^{2}}=\left\Vert D^{2}u\right\Vert _{L^{2}}$
and the Gagliardo-Nirenberg interpolation inequality $\left\Vert f\right\Vert _{L^{\infty}}^{2}\lesssim\left\Vert Df\right\Vert _{L^{\infty}}\left\Vert f\right\Vert _{L^{2}}$,
we see
\[
\left\Vert \partial_{t}u\right\Vert _{L^{\infty}}\le\Lambda^{2}\left\Vert D^{2}u\right\Vert _{L^{\infty}}\le Me^{-\omega t/2}
\]
for some $M>0$. Thus for $b>a>0$ we have 
\[
\left|u\left(x,b\right)-u\left(x,a\right)\right|\le\int_{a}^{b}\partial_{t}u\left(x,t\right)dt\le\int_{a}^{b}Me^{-\omega t/2}=\frac{2M}{\omega}\left(e^{-\omega a/2}-e^{-\omega b/2}\right)
\]
which converges to zero as $a,b\to\infty$; i.e. $u\left(t\right)$
is uniformly Cauchy as $t\to\infty$. Since we already know it converges
to $u_{\infty}$ on a subsequence, we obtain the full uniform convergence
$u\left(t\right)\to u_{\infty}$. Once again interpolating
\[
\left\Vert D^{j}\left(u-u_{\infty}\right)\right\Vert _{L^{\infty}}^{2}\lesssim\left\Vert u-u_{\infty}\right\Vert _{L^{\infty}}\left\Vert D^{2j}\left(u-u_{\infty}\right)\right\Vert _{L^{\infty}}
\]
and applying the uniform $C^{2j}$ bound, we get $C^{j}$ convergence
for all $j$.
\end{proof}
\bibliographystyle{halpha}
\bibliography{refs}

\begin{thebibliography}{Bak11}

\bibitem[And]{andrews-pinch}
Ben Andrews.
\newblock Pinching estimates and motion of hypersurfaces by curvature
  functions.
\newblock {\em J. Reine Angew. Math.}, 608:17--33.

\bibitem[Bak11]{bakerthesis}
Charles Baker.
\newblock The mean curvature flow of submanifolds of high codimension, 2011,
  arXiv:1104.4409.

\bibitem[GT83]{GT-EPDE}
David Gilbarg and Neil~S. Trudinger.
\newblock {\em Elliptic partial differential equations of second order}.
\newblock Grundlehren der mathematischen Wissenschaften. Springer-Verlag,
  Berlin, New York, 1983.

\bibitem[Lie96]{Lieberman-book}
Gary~M. Lieberman.
\newblock {\em Second order parabolic differential equations}.
\newblock World Scientific, Singapore, River Edge (N.J.), 1996.

\bibitem[LU64]{LU1}
O.~A. Ladyzhenskaya and N.~N. Uraltseva.
\newblock On h{\"o}lder continuity of solutions and their derivatives of linear
  and quasilinear elliptic and parabolic equations.
\newblock {\em Trudy Steklov Inst.}, 73:172--220, 1964.
\newblock (In Russian).

\end{thebibliography}

\end{document}